\documentclass[a4paper,10pt]{article}

\usepackage{times}
\usepackage[T1]{fontenc}
\usepackage[latin1]{inputenc}
\usepackage{amsmath,amssymb}
\usepackage{amsthm}
\usepackage{amscd}

\newtheorem{theorem}{Theorem}[section]
\newtheorem{lemma}[theorem]{Lemma}
\newtheorem{corollary}[theorem]{Corollary}
\newtheorem{proposition}[theorem]{Proposition}

\newtheorem{definition}[theorem]{Definition}

\theoremstyle{definition}
\newtheorem{example}[theorem]{Example}
\newtheorem{remark}[theorem]{Remark}

\numberwithin{table}{section}

\numberwithin{equation}{section}

\begin{document}
\title{A  functional calculus  and the complex conjugate of a matrix} 
\author{ Olavi Nevanlinna }
\maketitle

 \begin{center}
{\footnotesize\em 
Aalto University\\
Department of Mathematics and Systems Analysis\\
 email: Olavi.Nevanlinna\symbol{'100}aalto.fi\\[3pt]
}
\end{center}

\begin{abstract}  Based on Stokes' theorem we derive a non-holomorphic functional calculus for  matrices, assuming sufficient smoothness near eigenvalues, corresponding to the size of related Jordan blocks.   It is then applied to the  complex conjugation function $\tau: z \mapsto \overline z$.  The resulting matrix agrees with the hermitian transpose if and only if the matrix is normal.  Two other, as such elementary,  approaches to define the complex conjugate of a matrix yield the same result. 
 \end{abstract}
\bigskip



\section{Overview}

We discuss a simple non-holomorphic  functional calculus for  complex $n \times n$ - matrices $A\in \mathbb M_n(\mathbb C)$ with particular  focus in defining $\tau(A) \in \mathbb M_n(\mathbb C)$ for the complex conjugation function
\begin{equation}\label{tau}
\tau:  z \mapsto \overline z.
\end{equation}
We shall denote the resulting matrix $\tau(A)$ by $A^c$  and it agrees with  $A^*$  if and only if  $A$ is normal.    

Theorem 1.19 in [16]  says that  if   $\Omega$ is an open set either of $\mathbb R$ or $\mathbb C$, and if  $\varphi$ is  $n$-$1$ times continuously differentiable on the set $\Omega $, then $\varphi(A)$ is a continuous matrix function on the set of matrices $A\in \mathbb M_n(\mathbb C)$  with spectrum in $\Omega$. When $\Omega$ is open in $\mathbb C$, then $\varphi$ is in fact assumed to be holomorphic. Our non-holomorphic calculus shall be well defined for all  functions $\varphi$ in $C^{n-1}(\Omega)$, meaning that $\varphi$ has continuous partial derivatives up to oder $n-1$.

Recall first some basic facts on defining $\varphi(A)$ when $\varphi$ is holomorphic.
For $\varphi$ holomorphic in $\Omega$  and  continuous up to boundary, the Cauchy integral
\begin{equation}\label{Cauchy}
\varphi(z) = \frac{1}{2\pi i} \int_{\partial \Omega} \frac{\varphi(\zeta) }{\zeta - z} d\zeta,
\ \ z \in \Omega
\end{equation}
yields the basic definition for $\varphi(A)$,  if $\Omega$ contains the eigenvalues  and  the Cauchy kernel is  replaced by the resolvent
\begin{equation}
\zeta \mapsto (\zeta I - A)^{-1}.
\end{equation}
The integral can be evaluated by computing residues at the eigenvalues.     (On the history  of definitions of $\varphi(A)$, see e.g. Section 1.10 in [16], in particular  Frobenius [12]  used residues to define matrix functions).  Observe that  the use of Cauchy integral does not require transforming  the matrix into Jordan canonical form and representations for the resolvent can be available along the contour, without need to know the exact locations of eigenvalues. 


Assume now that  $\varphi \in C^1(\overline \Omega)$.  The Cauchy integral naturally still defines a holomorphic function in $\Omega$ but if $\varphi$ is not holomorphic,  an area integral is needed: 
\begin{equation}\label{Cauchy-Green}
\varphi(z) = \frac{1}{2\pi i} \int_{\partial \Omega} \frac{\varphi(\zeta) }{\zeta - z} d\zeta + \frac{1}{2\pi i} \int_\Omega \frac{\overline{\partial} \varphi(\zeta)}{\zeta-z} d\zeta \wedge  d\overline \zeta, 
\ \ z \in \Omega.
\end{equation}
This  is often referred to as Cauchy-Green  or Pompeiu's formula, [1], [15].  Here $\overline \partial$ denotes  the differential operator  $\frac{1}{2}( \frac{\partial}{ \partial \xi} + i \frac{\partial}{ \partial \eta} )$  when $\zeta=\xi +  i\eta$ is the complex variable. Notice that  $\varphi$ is holomorphic if and only if $\overline  \partial \varphi$ vanishes in $\Omega$.
Again  one can try to replace the kernel by the resolvent and in fact such approaches exist, with special restrictions on $\varphi$ and  on the operator to guarantee the convergence of the area integral, [9], [14],   [7], [8] [18], [4].
We focus here on matrices and therefore all singularities are poles, but in the presence of nontrivial Jordan blocks  higher order poles show up and these are not covered by (\ref{Cauchy-Green}).  However, for diagonalizable matrices $A=TDT^{-1}$   the formula yields immediately
\begin{equation}\label{integraalimaaritelma}
 \frac{1}{2\pi i} \int_{\partial \Omega} \varphi(\zeta)(\zeta I -A)^{-1}  d\zeta + \frac{1}{2\pi i} \int_\Omega {\overline{\partial} \varphi(\zeta)} (\zeta I -A)^{-1}  d\zeta \wedge  d\overline \zeta
=T \varphi(D) T^{-1}
\end{equation}
where $\varphi(D)=  {\rm diag}( \varphi(\lambda_j))$.  This is of course the natural matrix representing $\varphi(A)$ and, again, notice that the integrals provide approximative approaches for computing $\varphi(A)$ without   need to  know the Jordan canonical form of $A$.

In order to treat the higher order poles present with nontrivial Jordan blocks  we derive, based on Stokes' theorem,  a modification of (\ref{Cauchy-Green}), see Proposition \ref{genpompei}.
This then allows  a definition of $\varphi(A)$, provided $\varphi$ has required smoothness  around those eigenvalues.  To get a feeling on the required smoothness,  we return first to the holomorphic case.

Let $m_A(z)= \prod_{j=1}^N (z-\lambda_j)^{1+n_j}$ denote the minimal polynomial of $A \in \mathbb M_n(\mathbb C)$.   The following definition is  used  in  [16], following [13].

\begin{definition} The  function  $\varphi$ is said to be defined on the spectrum $\sigma(A)=\{\lambda_j\}_1^N$ of $A \in \mathbb M_n(\mathbb C)$  if  the values
\begin{equation}
\varphi^{(\nu)}(\lambda_j),   \ \ \ \nu =0, \dots,  n_j
\end{equation}
exist.   These are called the values of the function $\varphi$ on the spectrum of $A$.
\end{definition}
The matrix $\varphi(A)$ can then be obtained as $p(A)$ where $p(z)$is the Hermite interpolant agreeing with these values  at the spectrum. This definition of $\varphi(A)$   dates back to  Sylvester 1883 [22] and  Buchheim 1886 [2] and  it was shown by Rinehart 1955 [21]
that this agrees with the definition based on the Cauchy intgeral as well as some other approaches.  
Let 
\begin{equation}\label{Jordanblokki}
J=\begin{pmatrix} \lambda&1 & &  &\\
& \cdot & \cdot & & & \\
& & \cdot & \cdot & &\\
& & &\lambda&1& \\
& & & & \lambda\end{pmatrix} = \lambda I + S
\end{equation}
be an $m+1\times m+1$  Jordan block, then we have 
\begin{equation}
\varphi(J) = \sum_{\nu=0}^{m} \frac{\varphi^{(\nu)}(\lambda)}{\nu !} \ S^\nu.
\end{equation}

It is important to notice that here the derivatives are in two different meanings depending whether $\varphi$ is a holomorphic function of complex variable or a sufficiently differentiable function of a real variable - when the eigenvalues are all real.  Our aim  in this paper is to discuss  functions  of  a complex variable $\zeta= \xi + i \eta$
which have enough many continuous partial derivatives wrt $\xi$ and $\eta$.  The  question is  whether there is a natural way to condense the multitude of partial derivatives into one single number for each $\nu$.  It turns out  that if $\varphi \in C^{m}$, then we have based on Proposition \ref{genpompei}   likewise
\begin{equation}
\varphi(J)=\sum_{\nu=0}^{m} \frac{\partial^{\nu}\varphi(\lambda)}{\nu !} \ S^\nu,
\end{equation}
where we denote $\partial =\frac{1}{2}( \frac{\partial}{ \partial \xi} - i \frac{\partial}{ \partial \eta} )$.  Thus, the differential operator $\overline \partial$ does not appear at all and we are lead to define the needed interpolation data as follows.

\begin{definition}\label{uusdef}
Given a matrix $A \in \mathbb  M_n(\mathbb C)$  with minimal polynomial  
$m_A(z)= \prod_{j=1}^N (z-\lambda_j)^{1+n_j}$  the matrix $\varphi(A)\in  \mathbb  M_n(\mathbb C)$ is well defined if  $\varphi$ satisfies the following:  for  each $j$ there exists an $\varepsilon_j>0$ such that $\varphi \in C^{n_j}(B_j)$ where $B_j= \{\zeta \ : \ |\zeta-\lambda_j| \le \varepsilon_j\}$.

If $p$ is the Hermite interpolation polynomial satisfying for $j=1,\cdots, N$
\begin{equation}
p^{(\nu)}(\lambda_j) = \partial ^{\nu} \varphi(\lambda_j) \ \    Ê \text {    for    }\   \nu=0, \cdots, n_j,
\end{equation}
then we set   $\varphi(A):=p(A)$.

\end{definition}

\begin{remark}
This definition applies as such only to functions of a complex variable. However,  if $\varphi$ is a sufficiently smooth function of the real variable with derivative values $\varphi^{(\nu)}(\lambda_j)$  and if the eigenvalues are real  then this definition can still be used if  one extends $\varphi$ locally as follows:
\begin{equation}
\tilde \varphi (\lambda_j + \zeta): = \sum_{\nu=0}^{n_j} \frac{\varphi^{(\nu)}(\lambda_j)} {\nu !} \  \zeta^\nu + \varphi((\zeta + \overline \zeta)/2) - \sum_{\nu=0}^{n_j} \frac{\varphi^{(\nu)}(\lambda_j)}{\nu !} ((\zeta + \overline \zeta)/2)^{\nu}.
\end{equation}
Now we clearly have  
$$
\partial^{\nu}\tilde \varphi(\lambda_j)= \varphi^{(\nu)}(\lambda_j)
$$
so that the definition reproduces the usual matrix.
\end{remark}

\begin{example}
In general different extensions of a function defined on the reals may lead to different matrices. 
Let $\varphi(x)=x^2$ and extend it to complex $z$ in different ways:
$$
\varphi_0(z)=z^2, \ \ \varphi_1(z)=|z|^2,\ \  \varphi_2(z)= \overline z^2, \ \   \varphi_3(z) = ({\rm Re}(z))^2 , \  \  \varphi_4(z)= {\rm Re}(z^2)$$
Then with $A= \begin{pmatrix} \lambda&1\\ &\lambda \end{pmatrix}$ we have 
$
\varphi_0(A)= \begin{pmatrix} \lambda^2&2\lambda\\ &\lambda^2 \end{pmatrix},$ while
$$
\varphi_1(A)= \begin{pmatrix}| \lambda|^2&\overline \lambda\\ &|\lambda|^2 \end{pmatrix},  \ \ \ 
\varphi_2(A)= \begin{pmatrix} \overline \lambda^2&\\ & \overline \lambda^2 \end{pmatrix}, 
$$
and
$$
\varphi_3(A)= \begin{pmatrix}   ({\rm Re}(\lambda))^2 &  {\rm Re}(\lambda) \\ & ({\rm Re}(\lambda))^2 \end{pmatrix}, \ \ \
\varphi_4(A)= \begin{pmatrix}   {\rm Re}(\lambda ^2) &  \lambda \\ & {\rm Re}(\lambda^2) \end{pmatrix}.
$$
\end{example}
Consider now the complex conjugation function $\tau$ in (\ref{tau}).  We have $\partial \tau=0$ and $\overline \partial \tau=1$,  so that $\tau \in C^{\infty}(\mathbb C)$.  Applied to the Jordan block $J=\lambda I + S$ we therefore obtain 
\begin{equation}
\tau(J)= \overline \lambda I.
\end{equation}
 Hence, $\tau(A)$ is the {\it diagonalizable}  matrix $\tau(A)=T \overline D T^{-1}$  where $\overline D = {\rm diag}(\overline \lambda_j)$, if $A=T(D+N)T^{-1}$  is  the Jordan canonical form of $A$.   So,  evaluation of $\tau$ at $A$ looses information on possible Jordan blocks at multiple eigenvalues. As this may sound artificial, we  include in Section 2 an elementary derivation of $\tau(A)$  which is based on a list of simple requirements, which do hold for the adjoint $A \mapsto A^*$ ("hermitean transpose"), whenever $A$ is a normal matrix.  We shall see that these different  approaches lead to the same matrix $\tau(A)$ which we then name as  {\it the conjugate} of $A$.  In fact,  there is still a {\it third } reasoning also agreeing with these and that begins by considering {\it divided differences}, see Section 4.
 
 In Section 3 we discuss the  derivation of $\varphi(A)$ based on the  integral definition and   make some additional observations on the calculus. In  particular the mapping
 $$
 \varphi \mapsto \varphi(A)
 $$
 is  for all $A$ a continuous homomorphism  from sufficiently smooth scalar functions to matrices: 
 $$
 (\varphi \psi)(A) = \varphi(A) \psi(A).
 $$
 In contrast 
$$
A \mapsto \varphi(A)
$$
though well defined, is in general not bounded.  (On continuity of this mapping  for diagonalizabe matrices,  see [20]). 

In Section 4 we mention how the divided differences have to be modified so as to be consistent for functions of the form 
$$
z\mapsto p(z,\overline z)
$$
where $p$ is a polynomial in two variables.  Section 5 contains remarks on viewing $A^c$ as a polynomial of $A$ and Section 6   discusses   some properties of  approximatively computing the area integral  in the integral representation of $\tau(A)$.  In Section 7 we point out that defining  $ \rm{abs}:  z \mapsto {\sqrt{z \ \tau(z)}}$ gives a unique meaning for $ \rm{abs} (A)$ for square matrices, however,  necessarily with some "unwanted" properties.
At the end we mention how this approach is related to the non-holomorphic {\it multicentric calculus}, as presented in [19].


\section{Defining the conjugate $A^c$} 

The following discussion contains no nontrivial steps but for the sake of completeness   we   give   the details.  We begin with notation.  The complex conjugation  shall be denoted by $\tau$ as in (\ref{tau}).  If $A \in \mathbb M_n(\mathbb C)$ is an $n\times n$ complex matrix, then $\overline A$ stands for the matrix obtained by conjugating each element of $A$.

In addition to $\overline A$ there are two other meanings for complex conjugate of a matrix which are in common use.
 
\begin{example}  If $A$ is a bounded normal operator in a Hilbert space, then 
$$
A\mapsto A^*
$$ provides   the natural involution which is also an isometry.   Since  for  polynomials in two variables one has
$$
\|p(A,A^*)\| = \sup_{z\in \sigma(A)} |p(z,\overline z)|
$$  
the Stone-Weierstrass theorem yields immediately a  way to   define $\varphi(A)$ for all continuous functions.  
\end{example}

\begin{example}
If $A$ is a  diagonalizable matrix $A=T D T^{-1}$, then it is natural to set $\varphi(A)=T \varphi(D) T^{-1}$, where $\varphi (D)$ stands for the diagonal matrix diag$(\varphi(d_i))$.  This implicitly defines  a meaning for the conjugate  when applied to the function $\tau$:
\begin{equation}\label{perus}
A = TDT^{-1} \mapsto A^c= T \overline D T^{-1}.
\end{equation}
Since a normal matrix is unitarily similar to a diagonal one,  we have $A^*=A^c$ whenever $A$ is normal.  The main  focus in this  paper is in extending this calculus to nondiagonalizable matrices.
 \end{example}

\begin{remark}
There is a $C^{n+1}$-functional calculus for matrices  with spectrum on the unit circle (see [6], Prop. 4.5.13 for bounded operators in Banach spaces).  It is consistent with the other functional calculi when applied to diagonalizable matrices, but  not if there are nontrivial Jordan blocks. In fact, when applied to $\tau$ on the unit circle the calculus interpretes the function as restriction of $z\mapsto z^{-1}$  to the unit circle and yields $\tau(A)=A^{-1}.$  
\end{remark}

In order to extend $\tau$ to nondiagonalizable matrices we list some desirable properties which all hold for diagonalizable matrices with $\tau(A)=A^c$.   Here $A \in \mathbb M_n(\mathbb C)$  while $\{e_j\}_1^n$ denotes the standard basis of $\mathbb C^n$.  

\begin{equation}\label{i}
\tau(e_j e_j^*)= e_j e_j^*
\end{equation}
 
\begin{equation}\label{ii}
\tau(\alpha A) = \overline \alpha \  \tau(A),  \ \text{ for all complex } \ \alpha
\end{equation}

\begin{equation}\label{iii}
\tau(TAT^{-1}) = T \tau(A)T^{-1}  \ \text{ for all invertible  } \ T
\end{equation}

\begin{equation}\label{iv}
\tau(A+B)= \tau(A) + \tau(B),  \ \text{ whenever $A$ and $B$ commute }  
\end{equation}

\begin{equation}\label{v}
\tau(A) B = B \tau(A),  \ \text{ whenever $A$ and $B$ commute }
\end{equation}

\begin{equation}\label{vi}
\tau(\tau(A))=A.
\end{equation}
We begin by pointing out that (\ref{i})-(\ref{iv}) determine $\tau $ uniquely when applied to diagonalizable matrices.  In fact, writing a diagonal $D$ 
$$
D=  \sum_{j=1}^n \lambda  _j e_j e_j^*  
$$
yields using   (\ref{iv})  and (\ref{ii}) 
$$
\tau(D) = \sum_{j=1}^n \overline{\lambda  _j} e_j e_j^*  = \overline D.
$$
Then the uniqueness of $\tau$  for diagonalizable matrices follows from (\ref{iii}). 
Recall that two diagonalizable matrices commute if and only if they can be diagonalized with the same similarity transformation. Then it is clear that both (\ref{iv}) and (\ref{v}) hold for diagonalizable matrices.

We shall now show that if $A$ is not diagonalizable, (\ref{i})-(\ref{iv})  still determine $\tau $ uniquely,  but (\ref{vi}) no longer holds as $\tau(A)$ is always diagonalizable.   However, we begin with an illustrative example.

\begin{example}\label{eka}
Let 
$$
A=\begin{pmatrix}
\alpha&\gamma\\
0&\beta
\end{pmatrix}
$$
so that
$$A=TDT^{-1} = \begin{pmatrix}1& \frac{-\gamma}{\alpha-\beta}\\0&1\end{pmatrix}\begin{pmatrix}\alpha&0\\0&\beta\end{pmatrix}\begin{pmatrix}1& \frac{\gamma}{\alpha-\beta}\\0&1\end{pmatrix}.
$$
Since $\tau(D)= \overline D$ we obtain   
$$
\tau(A)= T \overline D T^{-1}= \begin{pmatrix}1& \frac{-\gamma}{\alpha-\beta}\\0&1\end{pmatrix}\begin{pmatrix}\overline \alpha&0\\0&\overline \beta\end{pmatrix}\begin{pmatrix}1& \frac{\gamma}{\alpha-\beta}\\0&1\end{pmatrix}
= \begin{pmatrix} \overline \alpha & \frac{\overline \alpha- \overline \beta}{\alpha-\beta} \gamma\\ 0 & \overline \beta\end{pmatrix}.
$$
Hence, if  $\beta = \alpha - t e^{i\theta}$ and $t \rightarrow 0$ the corner element tends to $e^{- 2i \theta} \gamma$. Thus,  $\tau$ is necessarily discontinouos at $\alpha=\beta$, but we shall see that the requirements above and  the  calculus based on the Stokes' theorem  both force the  corner element to vanish.
 \end{example}

Assume now that $A$ is given in the Jordan canonical form as $A=T(D+N)T^{-1}$, where $D$ is diagonal and $N$ may contain 1's  on the first  superdiagonal.  Because of  (\ref{iii}) we may assume for simplicity that $T=I$.  Write $A= D+N = \sum_k A_k$ where $A_k$ is a block diagonal matrix with one nonzero block  $J_k$  while the other blocks vanish.  Since all matrices $A_k$ commute we conclude that 
again by using (\ref{iv}) repeatedly that
$$
\tau(A)= \sum_k \tau(A_k).
$$
Hence we need to compute $\tau (A_k)$.   Since $A_k$ equals $J_k$ within  the nontrivial block and vanishes elsewhere we can decompose $A_k=D_k+N_k$    where $D_k$  is diagonal and takes the  value $\lambda  _k$ within the block  and vanishes elsewhere, while  $N_k$ has 1's  on the  first upper diagonal within the block and vanishes elsewhere.    Since $D_k$ and $N_k$ commute we have
$$
\tau(A_k)= \overline D_k + \tau(N_k).
$$
We show that $\tau(N_k)=0$.   Let $\alpha \not=0$ and consider  transforming $\alpha N_k$ into Jordan form.  However, $\alpha N_k$ is similar to $N_k$
$$
\alpha N_k= T(\alpha) N_k T(\alpha)^{-1}
$$ with $T(\alpha)$ diagonal with elements being suitable powers of $\alpha$. Hence we obtain by (\ref{ii})
$$
\overline \alpha \  \tau(N_k)= T(\alpha) \tau (N_k) T(\alpha)^{-1}.
$$
Since the left hand side   is a multiple  of $\overline \alpha$ while the right hand side has rational elements of $\alpha$  we must have $\tau (N_k)=0$.

We now summarize  the situation.
\begin{definition}
Let $A \in \mathbb M_n(\mathbb C)$ be given in the Jordan canonical form as
$$
A=T (D+N)T^{-1}
$$
 where $D$ is a diagonal matrix consisting of  the eigenvalues. Then  the conjugate $A^c$ is given by
$$
A^c =T\overline D T^{-1}.
$$
\end{definition}
The following holds.
\begin{proposition}
If $\tau:  \mathbb M_n (\mathbb C) \rightarrow \mathbb M_n (\mathbb C)$ satisfies  (\ref{i})  -  (\ref{iv}), then $\tau(A)=A^c.$   Furthermore $A \mapsto A^c$ satisfies (\ref{v}) always,  while (\ref{vi}) holds exactly when $A$ is diagonalizable.
\end{proposition}

All other claims have been discussed except  the commuting of $\tau(A)$ and $B$ whenever $A$ and $B$ commute. Notice the structural resemblance with Fuglede's 
theorem, according to which $A^*$ commutes with $B$ whenever $A$ is normal and commutes with $B$.   To prove (\ref{v})   directly from the definition is not difficult but we shall see in the  Section  5 that $A^c$ is a polynomial in $A$ and then the claim  follows immediately.

\begin{example}
Consider  a rank-one matrix $R= u v^*$.  Then, if $v^*u\not=0$ we get
$R^c = \frac{\overline{v^*u}}{v^*u} u v^*$ while, if $v^*u=0$, then $R^c=0$.  If  $u$ and  $v$ are unit vectors and $v^*u \not=0$ then $\|R^c\|=\|R\| =1$

\end{example}
\begin{example}
The conjugate of a companion matrix is not usually a companion matrix any more - except if  all roots are real and distinct.  In fact, consider  the   companion matrix $C$ of the polynomial $z^2-(\lambda_1 + \lambda_2) z + \lambda_1\lambda_2$. We have 
$$
C= \begin{pmatrix} & - \lambda_1 \lambda_2\\
1& \lambda_1+\lambda_2
\end{pmatrix} =  T DT^{-1}
$$
where $D= $diag$(\lambda_1, \lambda_2)$  and $T^{-1}$ the Vandermonde matrix
$$
T^{-1}=\begin{pmatrix} 1&\lambda_1\\
1& \lambda_2 \end{pmatrix}.
$$
This gives
$$
C^c= \frac{1}{\lambda_2-\lambda_1} \begin{pmatrix} 
\lambda_2 \overline \lambda_1 - \lambda_1\overline \lambda_2 & 
\lambda_2 |\lambda_1|^2 - \lambda_1 | \lambda_2|^2 \\
\overline \lambda_2 - \overline \lambda_1 & |\lambda_2|^2 - |\lambda_1|^2 \end{pmatrix}.
$$
If a polynomial has multiple roots then the companion matrix has maximal size Jordan blocks  and in particualr the conjugate being diagonalizable cannot be a companion matrix of a polynomial. 
For example,  with $(z-\lambda)^2$  we have 
$$
C= \begin{pmatrix} & - \lambda^2\\
1& 2 \lambda
\end{pmatrix} =  \begin{pmatrix} -\lambda&1\\
1& \end{pmatrix} 
\begin{pmatrix} \lambda & 1\\ & \lambda
\end{pmatrix}
 \begin{pmatrix} &1\\1&\lambda
\end{pmatrix}
$$
and then  $C^c$ itself is diagonal:
$$
C^c =  \begin{pmatrix} -\lambda&1\\
1& \end{pmatrix} 
\begin{pmatrix} \overline \lambda & \\ & \overline \lambda
\end{pmatrix}
 \begin{pmatrix} &1\\1&\lambda
\end{pmatrix} = \begin{pmatrix} \overline \lambda& \\ & \overline \lambda\end{pmatrix}.
$$

\end{example}

We may extend the Example 2.4 to block triangular matrices as follows.

\begin{example}
Let $A\in \mathbb  M_k(\mathbb C)$ and $B\in \mathbb  M_m(\mathbb C)$ be given and consider $M\in \mathbb M_{k+m}(\mathbb C)$ 
$$
M= \begin{pmatrix}  A & C\\
0&B \end{pmatrix}.
$$
If $\sigma(A) \cap \sigma(B)= \emptyset$, then the Sylvester equation $AX-XB=C$ has a unique solution $X$ such that 
$$
M=\begin{pmatrix} I&-X\\0&I \end{pmatrix}\begin{pmatrix}A&0\\0&B\end{pmatrix}\begin{pmatrix} I&X\\0&I \end{pmatrix}.
$$
Hence
$$
M^c=\begin{pmatrix} I&-X\\0&I \end{pmatrix}\begin{pmatrix}A^c&0\\0&B^c\end{pmatrix}\begin{pmatrix} I&X\\0&I \end{pmatrix}=\begin{pmatrix}  A^c & A^cX-XB^c\\
0&B^c \end{pmatrix}.
$$
\end{example}

Denote by $\rho(A)$ the spectral radius of $A$ and by $\varkappa(T)=\|T\| \|T^{-1}\|$ the condition number of the transformation bringing  $A$ to Jordan form. Then we  clearly have the following.
\begin{proposition}\label{sadeestimaatti} The conjugate  of $A \in \mathbb M_n(\mathbb C)$ satisfies
\begin{equation} \label{perusraja}
\rho(A) \le \|A^c \| \le \varkappa(T) \rho(A).
\end{equation}
\end{proposition}

We shall formulate a few other upper bounds for $A^c$  later, see  Propositions 3.7, 4.2 and 5.4.
The definition of conjugation implies immediately that conjugation  commutes with inversion: 
\begin{proposition}
If $A \in \mathbb M_n(\mathbb C)$  is nonsigular, then $(A^c)^{-1}=(A^{-1})^c$.
\end{proposition}
 

\section{ Calculus based on Stokes' theorem}

Pompeiu's formula expresses a $C^1$ - function inside a domain in two terms, the first of which extends the boundary information into a holomorphic function  and the second is an area integral  expressing the nonholomorphic part of the function.

Pompeiu's formula   is also called  Cauchy-Green formula.  To state it recall the {\it Wirtinger} differential operators
  
$$\partial =  \frac{\partial}  {\partial z} = \frac{1}{ 2}( \frac{\partial}  {\partial x} - i\frac{\partial}  {\partial y}),
$$
and 
 $$\overline{\partial} =  \frac{\partial}  {\partial \overline{z}} = \frac{1}{ 2}( \frac{\partial}  {\partial x} + i\frac{\partial}  {\partial y}).
 $$
Suppose $\varphi(z)$ is given in the form $\varphi(z)=w(z,\overline z) $.  Then one operates with these differential operators as if the variables $z$ and $\overline z$ were independent.  Here is a version  of the formula, see e.g.    [1], Theorem 2.1.   
 
\begin{proposition} Let $\Omega$ be a bounded domain with  boundary $\partial \Omega$ given by a finite number of piecewise smooth curves.
Assume $\varphi \in C^1(\overline{\Omega}) $. Then for $z\in \Omega$
\begin{equation}
\varphi(z)=\frac{1}{2\pi i} \int_{\partial \Omega} \frac{\varphi(\zeta)}{\zeta-z} d\zeta  +  \frac{1}{2\pi i} \int_{ \Omega} \frac{\overline \partial \varphi(\zeta)}{\zeta-z} d\zeta  \wedge d\overline \zeta .
\end{equation}
\end{proposition}

If  $\zeta=\xi +i \eta$ then $d\zeta\wedge \overline \zeta =  - 2i \ d\xi \ d\eta$, and  so in terms of  the Lebesgue area measure $\mu_2$,  the  formula looks as follows
$$
\varphi(z)=\frac{1}{2\pi i} \int_{\partial \Omega} \frac{\varphi(\zeta)}{\zeta-z} d\zeta-\frac{1}{\pi}\int_\Omega    \frac{\overline \partial \varphi(\zeta)}{\zeta-z} d\mu_2(\zeta).
$$
With nontrivial Jordan blocks the resolvent has higher order singularities and for that purpose we formulate a modification, which also contains an error term when a small circle is omitted around the singularity.


\begin{proposition}\label{genpompei}
Let $\Omega$ be a bounded domain with  boundary $\partial \Omega$ given by a finite number of piecewise smooth curves. Assume $\varphi\in C^m(\overline \Omega)$ with $m\ge 1$.   Let $z\in \Omega$ and $r$ small enough so that $D_r=\{\zeta \ : |\zeta-z|<r\} \subset \Omega$. Then  for $1\le n \le m+1$ 
\begin{align}\label{korkeatderivaatat}
& \frac{1}{(n-1)!}\ \partial ^{n-1} \varphi(z)  \\ 
& =  \frac{1}{2\pi i} \int_{\partial \Omega} \frac{\varphi(\zeta)}{(\zeta-z)^n} d\zeta  + \frac{1}{2\pi i} \int_{ \Omega \setminus D_r} \frac{\overline \partial \varphi(\zeta)}{(\zeta-z)^n} d \zeta  \wedge d \overline \zeta \\
& -  \sum_{1\le k \le (m+1-n)/2)} \frac{\partial^{n-1+k}\overline\partial ^k \varphi(z)}{(n-1+k)! \  k!} \ r^{2k}  + o(r^{m+1-n}) .
\end{align}
 
 \end{proposition}
\begin{proof}

Let $\Omega$ be a bounded domain with nice boundary and $\omega$ a complex 1-form on $\Omega$. Then Stokes'  theorem says:
$$
\int_{\Omega} d\omega = \int_{\partial \Omega} \omega.
$$
Take a fixed $z\in\Omega$ and choose a radius $r$ small enough  so that  
$$D_r=\{\zeta \ : |\zeta-z| <r \} \subset \Omega.
$$
From $\varphi \in C^m(\overline \Omega)$ we conclude that expanding $\varphi$ around $z$ there are coefficients $a_{jk}$ and a function $h$, depending on $z$,  such that
\begin{equation}\label{sileys}
\varphi(z+\zeta)= \sum_{0 \le j,k ; \ j+k \le m} a_{jk}\zeta^j {\overline \zeta}^k + h(\zeta),   \ \text{ where }  h(\zeta) = o(|\zeta|^m).
\end{equation}
We consider 1-forms of the form
$$
\omega = \frac{\varphi(\zeta)}{(\zeta-z)^n} d\zeta.
 $$
where $\varphi$ is  a smooth function. 
In the domain $\Omega_r= \Omega \setminus D_r$  the  1-form $\omega$ has no singularities and we can apply the Stokes' theorem. As the orientation of $\partial D_r$ is opposite of that of $\partial\Omega$ we obtain
$$
\int_{\Omega_r} d\omega = \int_{\partial \Omega_r} \omega =
 \int_{\partial \Omega} \omega - \int_{\partial D_r} \omega.
 $$
Here the boundary integral over $D_r$ is easy to compute and as it has a limit when $r\rightarrow 0$,  then so has the area integral over $\Omega_r$.
Consider therefore   the integral over $\partial D_r$. Parametrizing the boundary by $ z+re^{i\theta}$ we have
$$
\int_{\partial D_r}\omega= \int_0^{2\pi} \frac{\varphi(z+re^{i\theta})} {r^n e^{in\theta}}  r ie^{i\theta}d\theta
$$
and substituting here $\varphi$ from (\ref{sileys}) gives
\begin{align}
\int_{\partial D_r}\omega= & \  i\sum_{j,k} a_{jk} \int_0^{2\pi}
r^{j+k+1-n}e^{i(j-k+1-n)\theta} d\theta +o(r^m) r^{1-n} \\
= &2\pi i \  \sum_{k} a_{n-1+k,k} \ r^{2k}  + o(r^{m+1-n})\\
= &2\pi i \  \sum_{k} \frac{\partial^{n-1+k}\overline\partial ^k \varphi(z)}{(n-1+k)! \ k!} \ r^{2k}  + o(r^{m+1-n}).
\end{align}
 
 Since 
$$
d\omega = \frac{\overline \partial \varphi(\zeta)}{(\zeta -z)^n} d\overline \zeta \wedge d\zeta = - \frac{\overline \partial \varphi(\zeta)}{(\zeta -z)^n} d\zeta \wedge d\overline \zeta,
$$
the claim follows.
  
\end{proof}

Assume now that $A\in \mathbb M_n(\mathbb C)$ is given and $\Omega $ is  as in  Proposition 3.2 including the eigenvalues $\{\lambda_j\} =Ê\sigma(A) \subset \Omega$.
 We aim to show that the integral formula 
 \begin{equation}\label{kaava}
 \varphi(A)= \frac{1}{2 \pi i} \int_{\partial \Omega} \varphi(\zeta) (\zeta I- A)^{-1} d\zeta
 + \frac{1}{2 \pi i} \int_{\Omega} \overline{\partial} \varphi (\zeta) (\zeta I- A)^{-1} d\zeta
 \wedge d \overline \zeta
 \end{equation}
is consistent with  the $\varphi(A)$ as determined in Definition 1.2. For holomorphic functions Chapter 6 in [17] covers carefully different  approaches in defining $\varphi(A)$.  It is clear that, in the same way as in the holomorphic functional calculus,  the integral formula implies
$$
\varphi(A)= T \varphi(J)T^{-1}
$$
where $A=TJT^{-1}$  denotes the Jordan canonical form of $A$.  The boundary integral is well defined  as the resolvent is bounded along the boundary, but the area integral needs special consideration as the  resolvent has poles inside $\Omega$.  So,   if we write the $J=\oplus_j J_j$, then we can further reduce the discussion to each Jordan block $J_j$  separately as they  operate independently of each others. Let therefore
\begin{equation}\label{Jordanblokkiuudestaan}
J_j=\begin{pmatrix} \lambda_j&1 & &  &\\
& \cdot & \cdot & & & \\
& & \cdot & \cdot & &\\
& & &\lambda_j&1& \\
& & & & \lambda_j\end{pmatrix} = \lambda_j I + S
\end{equation}
 be an $1+n_j \times 1+ n_j$ Jordan block, so that
$$
(\zeta I - J_j)^{-1} = \sum_{\nu=0}^{n_j} \frac{1}{(\zeta-\lambda_j)^{\nu +1}} \ S^\nu.
$$
Let $D_r$ denote as before a disc centered at $\lambda_j$ . Substituting  the resolvent  into (\ref{kaava}) yields immediately the existence of the area integral and we have
\begin{align}
 & \frac{1}{2 \pi i} \int_{\partial \Omega} \varphi(\zeta) (\zeta I- J_j)^{-1} d\zeta
 + \lim_{r\rightarrow 0} \frac{1}{2 \pi i} \int_{\Omega \setminus  D_r} \overline{\partial} \varphi (\zeta) (\zeta I- J_j)^{-1} d\zeta \wedge d \overline \zeta \\
 =  & \sum_{\nu=0}^{n_j}  \frac{1}{\nu !} \partial ^\nu \varphi(\lambda_j) S^{\nu}.
\end{align}
We formulate this  as a separate result.

\begin{proposition}
Let $A \in \mathbb M_n(\mathbb C)$ with $n\ge 2$ be given and assume
 $\Omega$ is  a bounded open simply connected set with a piecewise smooth boundary such that $\Omega$ contains the eigenvalues  $\sigma(A)$ of $A$.     If $\varphi \in C^{n-1} (\overline \Omega)$ is given, then  $\varphi(A)$ as defined by Definition 1.2 equals the matrix obtained from  (\ref{kaava}).
\end{proposition}

\begin{remark}
Recall that  in Definition 1.2  we only requested the partial derivatives to exist and be continuous locally around eigenvalues.    So, if $B_j=\{\zeta \ : |\zeta -\lambda_j| \le \varepsilon_j\}$  are small enough so that  they  do not overlap each others,  and $\varphi$ is only assumed to  be defined in these discs,  the following is possible.  First take a simply connected $\Omega$ containing all these discs and satifying the assumptions in Proposisition 3.3.  Then you could extend $\varphi$ smoothly from $B_j$'s to $\Omega \setminus (\cup_j B_j)$ e.g. by using suitable cut-off functions.  Then one could apply (\ref{kaava})  and since the result does not  depend on $\Omega$ as long as it contains  all eigenvalues, once could shrink it to separate discs $B_j$, in the same way as one would shrink a contour integral of a rational  function to surround each pole separately.   
Thus, the only case where the Definition 1.2 works which is not covered by the calculus obtained by the Cauchy-Green formula, is when  $n_j=0$, as then  no  partial derivatives are asked for at all in the definition,  while the formula  to work  asks for   $\varphi \in C^1$.\end{remark}
 
\begin{corollary}
Let $A\in \mathbb M_n(\mathbb C)$ be given and  $\Omega$ an open set containing $\sigma(A)$.  Then the mapping  
$$\varphi \mapsto \varphi(A)$$
from $C^{n-1}(\overline \Omega)$ to $\mathbb M_n(\mathbb C)$ is a homomorphism.
\end{corollary}
\begin{proof}
The linearity is obvious from the definition:
$$
(\alpha \varphi + \beta \psi)(A)= \alpha \varphi(A) + \beta \psi(A),
$$
while  the product
$$
(\varphi \psi)(A)= \varphi(A) \psi(A)
$$
 can be concluded either by transforming  $A$ into Jordan canonical form and  using the Definition 1.2 together with computing the partial derivatives required, or directly as such from the integral formula  in Proposition 3.2.

\end{proof}


From Definition 1.2 it is clear that $\varphi(A)$ agrees with any $f(A)$ where $f$ is holomorphic in the neighborhood of   the spectrum  of $A$ and agrees with the same interpolation data.   This can be used to derive a bound for $\varphi(A)$.  Consider first the  following interpolation problem.

  Given $\{\lambda_j\}_{j=1}^N \subset \mathbb D$,  nonnegative integers $n_j$ and complex numbers $w_{j, \nu}$ to find a   bounded holomorphic function $f$ in the unit disc such that it satisfies the interpolation conditions
\begin{equation}\label{NPmult}
f^{(\nu)}(\lambda_j) = w_{j,\nu} \ {\text  for } \ j=1,\dots, N; \ \nu=0, \dots, n_j.
\end{equation}
Since  there exists a unique  Hermite interpolation polynomial solving this problem,  the interest is in finding a solution with a small norm in $H^\infty(\mathbb D)$.   It was shown in [Earl], Theorem 3  (see  [5], Theorem 9.6)  that there exists a unique function $h$ with smallest norm and that $h$ is a constant multiple of a finite Blaschke product with at most 
$
\sum_{j=1}^N (1+n_j) -1
$ factors.  By scaling the  matrix $A$  suitably, we may assume that  $\|A\| \le 1 $ and $\rho(A) <1$. 

  \begin{proposition}
Given $A\in \mathbb M_n(\mathbb C)$ such that $\sigma(A) \subset \mathbb D$ and $\|A\| \le 1$ with minimal polynomial $m_A(z) = \prod_{j=1}^N (z-\lambda_j)^{1+n_j}$.  Let     $\varphi$ satisfy the conditions of Definition \ref{uusdef}, then 
\begin{equation}
\| \varphi(A)Ê\| \le \sup_{|z| \le 1} |h(z)|,
\end{equation}
where $h$ is the unique minimal norm solution of  the interpolation problem {\rm (\ref{NPmult})  }with $w_{j,\nu}= \partial^\nu \varphi(\lambda_j)$.

\end{proposition}
\begin{proof}
Since  the spectrum of $A$ is in the open unit disc $h(A)$ is well defined in the holomorphic functional calculus and on the other hand $h$, being a multiple of a finite Blaschke product, can be approximated in the closed unit uniformly with polynomials $p_k$ so that $p_k(A) \rightarrow h(A)$.  Then the result follows from the von Neumann result
$
\|p(A)\| \le \sup_{|z| \le 1} |p(z)|.
$
\end{proof}

For diagonalizable matrices we can formulate a  more  quantitative estimate, by  borrowing some deep results from function theory to our finite interpolation problem.    Recall that 
$\Lambda= \{\lambda_j\}_{j=1}^\infty \subset \mathbb D$ is called  an {\it interpolation sequence} if the interpolation problem
\begin{equation}
f(\lambda_j) =  w_j,    \  \text {  for   } \ j=1,2,\dots
\end{equation}
has  for all $\{w_j\} \in l^\infty$ a  solution in $H^{\infty}(\mathbb D)$.  It is known [3] that this is equivalent  with the following separation  property:
\begin{equation}\label{sep}
\delta_\Lambda := \inf_n \prod_{m\not =n } |\frac {\lambda_m-\lambda_n}{1-  \lambda_m \overline\lambda_n}| >0.
\end{equation}

\begin{proposition}
Given a diagonalizable $A\in \mathbb M_n(\mathbb C)$ such that $\sigma(A) \subset \mathbb D$ and $\|A\| \le 1$,   assume     $\varphi$  to  be defined at $\sigma(A)$. 
Denoting $$
\delta  :=\min_{\lambda \in \sigma(A)} \prod_{\lambda\not=\lambda_j\in \sigma(A)} |\frac {\lambda_j-\lambda}{1-  \lambda_j \overline\lambda}| 
$$ we have
\begin{equation}
\| \varphi(A)\| \le \frac{4}{\delta ^2} \max_{\lambda \in  \sigma(A)} | \varphi(\lambda)|.
\end{equation}
\end{proposition} 

\begin{proof}
We can simply  complete the finite  set of eigenvalues $\{\lambda_j\}_1^N$ into an interpolation sequence $\Lambda$ by setting e.g.   $\lambda_{N+k}= 1-\varepsilon ^k$ for $k=1, 2, \dots.$
Then one checks easily that  as $0< \varepsilon \rightarrow 0$ we have $\delta_L \rightarrow \delta.$
Now  there exists a function $f \in H^{\infty}(\mathbb D)$ such that 
$$
f(\lambda_j) = \varphi(\lambda_j)  \text  { for }  \lambda_j \in \sigma(A)
$$
while we may require for  example $f(\lambda_{n+k}) = 0$.    By Theorem 2 in [10]  there exists a Blaschke product $B$  and a complex $\gamma$ satisfying $|\gamma| < 4/\delta^2 \maxÊ|\varphi(\lambda_j)|$ such that 
$\gamma B(z)$ satisfies these interpolation conditions.  Since  the spectral radius $\rho(A) <1$  we conclude that $\gamma B(A)$ is well defined in the holomorphic functional calculus and we have
$$
\|\gamma B(A)\| \le |\gamma| \sup_{|z| \le \|A\| } |B(z)|
$$
whenever $\|A\| <1$, as $B$ is holomorphic in the open disc.   The case $\|A\| =1$ follows then by approximating $A$ by $\theta A$ with $\theta<1$. 

\end{proof}

\section{Divided differences modified}

By Stone-Weierstrass theorem functions $\varphi \in C(\overline \Omega)$ can be approximated uniformly by  functions  of the form $ z \mapsto p(z,\overline z)$ where  $p$ is a polynomial of two variables.  The functional calculus derived in the previous section could also be  based on  divided differences. In fact,  after  a Schur transformation into upper triangular form the elements can be given recursively using divided differences. 

Let $\Lambda=\{\lambda_j\}$ be a finite  ordered set of points, such that if $i<j$ and $\lambda_i= \lambda_j$, then $\lambda_k= \lambda_i$ for all $i\le k \le j$.  Then given a function $\varphi: z \mapsto \varphi(z)$  one sets $\varphi[\lambda_i]:=\varphi(\lambda_i)$ and  if $\lambda_i \not= \lambda_j$ 
\begin{equation}
\varphi[\lambda_i, \dots, \lambda_j]: = \frac{\varphi[\lambda_{i+1},\dots,\lambda_j] - \varphi[\lambda_i,\dots, \lambda_{j-1}]}{\lambda_j -\lambda_i}.
\end{equation}
In accordance with Definition 1.2 we modify the usual definition of the divided differences at multiple points:    

\begin{equation}\label{eijatkuva}
\varphi[\lambda_i, \dots, \lambda_j] := \frac{1}{(j-i)!} \  \partial ^{j-i} \varphi(\lambda_i),  Ê \text   { when  }    \lambda_i = \lambda_j . 
\end{equation}

For holomorphic $\varphi$ the divided differences and the functional calculus are connected by the Newton interpolation formula or   by the {\it Opitz theorem,} e.g. [16], which says that if 
\begin{equation}
A = \begin{pmatrix}  \lambda_1&1& &  &\\
 & . & . & &\\
 & & . & .&\\
 & & & \lambda_{n-1}& 1\\
&  &  &   &  \lambda_n
\end{pmatrix}
\end{equation}
then $\varphi[\lambda_i, \dots, \lambda_j] = e_i^* \varphi(A) e_j$.  Proceeding mechanically one can verify that this statement also holds under the convention (\ref{eijatkuva}) when $\varphi(A)$ is defined by Definition 1.2.  

Since the non-holomorphic functional calculus is in general not continuous, some key properties of divided differences are necessary lost.   In particular, anything elegant comparable to Hermite-Genocchi representation cannot exist, as  all higher partial derivatives of $\tau$  vanish,  while already $\tau[\lambda_1,\lambda_2,\lambda_3]$  is not only discontinuous but unbounded. For example
$$
\tau[\lambda, \lambda+\varepsilon, \lambda+i\varepsilon] = \frac{1}{1-i}\frac{\overline \varepsilon}{\varepsilon^2}
$$
and in particular we conclude that {\it the mapping $A\mapsto A^c$ is necessarily unbounded for $n\ge 3$}, independently of how we define the mapping at the nontrivial Jordan blocks.
 
Above we motivated the covention (\ref{eijatkuva}) as it is the one which agrees with the calculus we obtained from the Stokes' theorem.  However, we could  take an {\it independent } beginning and arrive to the same definition.
 To that end,  consider  composed functions $f \circ \tau:  z \mapsto f(\overline z)$ with $f$ holomorphic.  Then with $\lambda_1\not=\lambda_2$ we have 
\begin{equation}
(f\circ \tau)[\lambda_1,\lambda_2] = \frac{ f(\overline \lambda_2)-f(\overline \lambda_1)}{\overline \lambda_2- \overline \lambda_1} \  \frac{\overline \lambda_2- \overline \lambda_1} {\lambda_2- \lambda_1}  \ = \ f[\overline \lambda_1, \overline \lambda_2]  \ \tau[\lambda_1,\lambda_2].
\end{equation} 
If $\lambda_2 \rightarrow \lambda_1$,  then $f[\overline \lambda_1, \overline \lambda_2] \rightarrow f'(\overline \lambda_1)$, but  the second term  
 $
  \tau[\lambda_1,\lambda_2]= \frac{\overline \lambda_2-\overline \lambda_1}{\lambda_2-\lambda_1}
 $
can take any value of modulus 1.  Since we cannot make $\tau[\lambda_1,\lambda_2]$ continuous at $\lambda_1=\lambda_2$  the  natural  things to guarantee  are that  

(i)  $\lambda \mapsto \tau[\lambda,\lambda]$ is continuous and that  

(ii) the choice is optimal in some sense.   

Defining
\begin{equation}\label{maarittely}
\tau[\lambda,\lambda]:=0
\end{equation}
we obtain  the {\it unique minimizer of the maximum error in evaluating the divided difference $\tau[\lambda+\varepsilon_1, \lambda +\varepsilon_2] $
 when $\varepsilon_i$ are small but not known exactly. }    The same reasoning also leads to 
 $\tau[\lambda, \dots, \lambda]=0$.

Using the definition  (\ref{maarittely}) one can  work out the divided differences for functions $z\mapsto z^k \overline z^m$ in order  to verify that (\ref{eijatkuva})  follows from this convention for polynomials of the form $\sum_{k,m} \alpha_{k,m} z^k \overline z^m$.  The other essential rule for operation is, that when operating with $\partial$ to the function   one keeps $\overline z$ as a fixed independent constant  so that in particular $\partial (z^k \overline z^m) = k z^{k-1} \overline z^m.$  For example,
$$
z^k\overline z^m[\lambda_1,\lambda_2]= z^k[\lambda_2] \ z^m[\overline \lambda_1,\lambda_2] \ \tau[\lambda_1,\lambda_2] + z^k [\lambda_1,\lambda_2] \  \overline z^m[\lambda_1]
$$
while  
$
z^k \overline z^m[\lambda,\lambda]= z^k[\lambda,\lambda ] \ \overline z^m[\lambda]= k \lambda^{k-1} \overline \lambda^m.
$



\begin{remark}
Suppose now  that $\varphi$ has continuous partial derivatives $\partial^\nu \varphi$ at the eigenvalues of a matrix $A$, corresponding to the size of  the related Jordan blocks.   For {\it the  Schur-Parlett algorithm } for  holomorphic $\varphi$  we refer to [16]  Section 4.6 and  Chapter 9.  The basic formulas still hold but at treating clusters of close eigenvalues cannot be done in the same way.  Recall that one starts  by computing a Schur decomposition $A=QTQ^*$ where $T$ is upper triangular and  $Q$ unitary.    Then one should  reorder  $T$ into another triangular matrix such that  the eigenvalues near each other  can be blocked into same diagonal blocks.   We assume that either  eigenvalues in each block can be viewed identical or all are  well separated from each others.   Thus, suppose then that $T$  is a triangular matrix $T=(T_{ij}) $  where the diagonal blocks are of the form
\begin{equation}\label{nilpoblokki}
T_{ii} = \lambda_i I_i +N_{ii}
\end{equation}
where $I_i$ is the identity matrix of appropriate size and $N_{ii} $ is  nilpotent strictly upper triangular  matrix.

Denote $\varphi(T) = (\Phi_{ij})$  so that the evaluation of $\varphi(T)$ reduces to computation of $\Phi_{ij}$'s. The  reucursive  computation of these blocks is done by solving repeatedly the following Sylvester equations  [16] for $i<j$
\begin{equation}\label{SP}
T_{ii}\Phi_{ij}-\Phi_{ij}T_{jj} = \Phi_{ii}T_{ij}-T_{ij} \Phi_{jj} + \sum_{k=i+1}^{j-1} (\Phi_{ik}T_{kj}-T_{ik} \Phi_{kj}).
\end{equation}
This system of equations is nonsingular if and only if the  diagonal blocks $T_{ii}$ and $T_{jj}$ contain no common eigenvalue - which is assumed  by  the proper preprocessing.
Thus,  the recursion to start one needs to evaluate  the diagonal blocks $\Phi_{ii}$.
If all eigenvalues in  the block $T_{ii}$ are different, then one proceeds as usual, using Parlett recurrence,  see Algorithm 4.13 in [16].   However,   with  blocks of the form
(\ref{nilpoblokki})  we need to modify the procedure.  While for holomorphic functions one can apply the Taylor series, even when the diagonal elements are not exactly equal, here, due to the discontinuity, we must consider them equal and use the following expansion instead:
If  $T=\lambda I + N \in \mathbb  M_n(\mathbb C)$, and  $N$  is  nilpotent,  then  \begin{equation}
\varphi(T)= \sum_{\nu=0}^{n-1} \frac{\partial ^\nu \varphi(\lambda)}{\nu !} N^\nu.
\end{equation}
 \end{remark}
 
We can obtain  yet another bound for $A^c$ using  the Parlett recursion. 
\begin{proposition}
Let $A\in \mathbb M_n(\mathbb C)$. If $\sigma(A)=\{\lambda\}$, then $\|A^c \| \le \|A\|$.  Otherwise, if $\sigma(A)= \{\lambda_i\}$,  let $\delta \le |\lambda_i - \lambda_j|$ for $i \not=j$.  Then
\begin{equation}
\|A^c\| \le n \ (2 \ \frac{\|A\|}{\delta} +1) ^{n-2} \|A\|.
\end{equation}

\end{proposition}

\begin{proof}
 By Schur's theorem  it sufficces to establish this for upper triangular matrices. 
If $\sigma(A)=\{\lambda\}$, then $A^c=\overline \lambda I$ and $\|A^c\|= \rho(A) \le \|A\|.$ 
Otherwise,  $A =(a_{ij}) \in \mathbb M_n(\mathbb C)$ is  upper triangular with a nonconstant diagonal.  We assume that the diagonal elements $a_{ii}= \lambda_i$ have been ordered  as in the definition of divided differences above.  Thus either  $a_{ii} =a_{jj}$ or  $|a_{ii} - a_{jj}| \ge \delta.$   Further,  $\|A\| \ge |a_{ij}|$.  Denoting $A^c = (b_{ij})$  we have from $b_{ii}= \overline a_{ii}$ that $|b_{ii}| \le \|A\|.$  One can now work  from the diagonal towards to upper right corner  using  the  Parlett's recurrence  
$$
b_{ij} = \frac{b_{ii}-b_{jj}}{a_{ii}-a_{jj}} \ a_{ij}+ \frac{1}{a_{ii} - a_{jj}}
\sum_{k=i+1}^{j-1}(b_{ik}a_{kj}-a_{ik}b_{kj})
$$
and  utilizing the  convention $\tau[\lambda,\cdots,\lambda]=0$, 
 to get   $$
 |b_{ij}| \le \|A\| (\frac{2 \|A\|}{\delta}+1)^{j-i-1}.
 $$
 Thus
 $\|A^c\| \le n  \max \{ |b_{ij}| \} \le n \ \|A\| (\frac{2 \|A\|}{\delta}+1)^{n-2}.
  $

 \end{proof}

\begin{example}
Let $X$, $Y$ be  real $n\times n$-matrices with elements uniformly distributed on $[-1/2, 1/2]$ and set $A=X+iY$.  Small scale experiments with MATLAB   for $n$ below 1000 seem to indicate that roughly  in the average
$$
\|A^c \| \sim 0.7  \log(n) \|A\|
$$
while $\varkappa(T)$  stays typically above $0.5 n$ and $\|A\| $ below $2 \rho(A)$.
\end{example}

 
 \section{The conjugate as a polynomial}
 
We shall now apply the calculus  to the conjugation function $\tau$.  Let $A\in \mathbb M_n(\mathbb C)$ be given and $\Omega$ such that it includes $\sigma(A)$, the set of eigenvalues of  $A$.     Then we set following Proposition 3.3
\begin{equation}\label{integraalikonjug}
\tau(A)= \frac{1}{2\pi i} \int_{\partial \Omega} \overline \zeta \  (\zeta I-A)^{-1} d\zeta+
\frac{1}{2\pi i} \int_{ \Omega}  (\zeta I-A)^{-1} d\zeta  \wedge d\overline \zeta.
\end{equation}
 Let $m_A$ denote the minimal polynomial of $A$ and put $m= {\rm deg} (m_A)$. 
Then set
$$
q(\zeta,z) = \sum_{j=0}^{m-1} q_j(\zeta) z^j = (m_A(\zeta) - m_A(z))/(\zeta -z),
$$  
so that deg $q_j = m-1-j$.  In this notation the resolvent  becomes, as $m_A(A)=0$
\begin{equation}\label{resolvent}
(\zeta I-  A)^{-1}= \frac{1}{m_A(\zeta)}\sum_{j=0}^{m-1} q_j(\zeta) A^j .
\end{equation}

\begin{proposition}\label{pompeikautta}
The  integrals in  (\ref{integraalikonjug})  are well defined and 
$$
\tau(A)=A^c= \sum_{j=0}^{m-1} \alpha_j A^j
$$
where
$$
\alpha_j=\frac{1}{2\pi i} \int_{\partial \Omega} \frac{\overline{\zeta} \ q_j(\zeta)}{m_A(\zeta)} d\zeta  +  \frac{1}{2\pi i} \int_{ \Omega} \frac{ q_j(\zeta)}{m_A(\zeta)} d\zeta  \wedge d\overline \zeta .
$$

If $B$ commutes with $A$, then $B$ commutes with $A^c$.
\end{proposition}

\begin{proof}

Substitute the resolvent from (\ref{resolvent}) into  (\ref{integraalikonjug}).  Then  the expressions for $\alpha_j$ appear.  Observe that the sum of these two integrals has a constant value independent of $\Omega$, as long as $\Omega$ contains the eigenvalues.    That the integrals are well defined and $\tau(A)=A^c$ follow from  both  Proposition 3.2 and from Proposition 3.3.

 Since $\tau(A)$ is a polynomial of $A$,  it is clear that $B$ commutes with $\tau(A)$ whenever it commutes with $A$.

 \end{proof}

 Since $A^c$ is  a polynomial of $A$ we may  introduce the following notion.
\begin{definition}
Given $A\in \mathbb M_n(\mathbb C)$ we set $\tau_A$  the conjugating polynomial of $A$ by
$$
\tau_A(z) = \sum_{j=0}^{m-1} \alpha_j z^j,
$$
where the coefficients $\alpha_j$ are given in Proposition \ref{pompeikautta}.    

\end{definition}

Since $\tau_A(A)=A^c$ is the value of $z\mapsto \overline z$ we can write  it using the Lagrange interpolation formula, whenever $A$ has distinct eigenvalues.  In fact, if $\{\lambda  _j\}$ denote the eigenvalues,  and $\delta_k$ denote the interpolation polynomials taking value 1 at $\lambda_k$ and vanishing at the other eigenvalues,  then  clearly
$$
\tau_A(z)= \sum_{j=1}^n \delta_j(z) \overline {\lambda}_j
$$
 satisfies $\tau_A(A)=A^c$.  More generally,   if $A$ has multiple eigenvalues with minimal polynomial 
 \begin{equation}\label{minpol}
 m_A(z)= \prod_{j=1}^N (z-\lambda_j)^{1+n_j}
 \end{equation}
 where $\sum_{j=1}^N (1+n_j) = m \le n$ then we obtain the conjugating polynomial by requiring the interpolation polynomial to have vanishing derivatives up to order $n_k$  at $\lambda_k$, whenever $n_k >0$.   We formulate this as a separate result.
 
 \begin{proposition}  If  the minimal polynomial of $A$ is given by (\ref{minpol}), then the conjugating polynomial $\tau_A$ is  the unique  polynomial of degree $m-1$ satisfying the interpolation conditions for all $j$:
 $$
 \tau_A(\lambda_j) = \overline \lambda_j    
 $$
 and  whenever $n_j>0$,  one has  for $i=1, \dots, n_j$
 $$
\tau_A ^{(i)}(\lambda_j)=0.
$$
\end{proposition}

\begin{proof}  Since we already know that $A^c$ is always diagonalizable, then the extra conditions on  derivatives follow immediately from the invariance over similarity transformations.
\end{proof}

In addition to Propositions \ref{sadeestimaatti} and 4.2 we can  bound $A^c$ using the conjugating polynomial.

\begin{proposition}\label{vNestimaatti} For $A\in \mathbb M_n(\mathbb C)$ we have
$$
\|A^cÊ\| \le \sup_{|z| \le \|A\|} |\tau_A(z)|.
$$
\end{proposition}
\begin{proof}
This follows from the von Neumann result for contractions in Hilbert spaces, applied to the polynomial $\tau_A$.
\end{proof}

\begin{example}
Let $\lambda\not=0$ but $|\lambda| < 1$, $n\ge 2$  and  put 
$$
B_0= \lambda e_1 e_1^* + \sum_{j=2}^{n-1} e_j e_{j+1}^*
$$
and
$$
B= B_0 + e_1e_2^*.
$$
For both matrices the  conjugating polynomial is $\tau_B(z) = \overline \lambda ({z}/{\lambda})^{n-1}$. We obtain
$$
B_0^c = \tau_B(B_0)= \overline \lambda e_1e_1^*
$$ while
$$
B^c = \tau_B(B)= \overline \lambda e_1e_1^*+  \overline \lambda e_1 \sum_{k=2}^{n} \lambda^{1-k} e_{k}^*.
$$
Further, $\rho(B_0)=\rho(B)=|\lambda|$. From Proposition \ref{sadeestimaatti} we have,  in general
$$
\rho(A) \le \|A^cÊ\| \le  \varkappa(T) \rho(A)
$$
and since $B_0$ is already in Jordan form,  indeed, $\|B_0^c \| = \rho(B_0)=|\lambda|$.  If $T= I - e_1 \sum_{j=1}^{n-1} \lambda^{-j} e_{j+1}^*$, then $B=T B_0 T^{-1}$  and so $B^c = T B_0^c T^{-1}$.
But  since $\varkappa(T) = (1+o(1)) |\lambda|^{2(1-n)}$ as $\lambda\rightarrow 0$, the estimate 
$$ \|B^c\| \le  \varkappa(T) \rho(B)
$$
yields $ \|B^c \|  \le (1+o(1)) |\lambda|^{3-2n}$ while in fact $\|B^c\| =(1+o(1)) |\lambda|^{2-n}$.  On the other,  hand  Proposition \ref{vNestimaatti} gives
 $$
 \|B^c \| \le (1+o(1)) |\lambda|^{2-n},
 $$ which is now sharp while the same inequality for $B_0$ is  very pessimistic.
 Finally, notice that $B^{cc}=  \frac{\lambda}{\overline \lambda} B^c$ so that for $n>2$ and $\lambda$ small we have 
 $\|B^{cc} \| >> \| B \|$.
\end{example}

Thus, in the previous example Proposition 5.4 is quite  sharp  while Proposition 4.2 also gives qualitatively right order of singularity  (i.e. $\mathcal O( \delta^{2-n})$ as $\delta\rightarrow 0$).  Estimates based on the conjugate polynomials are in general not very useful, as the following example shows.

  \begin{example}
Let $\lambda \not= \mu$ and $\varepsilon $ small enough so that  $|\varepsilon| <|\lambda-\mu|$. Consider
$$
A(\varepsilon) = \begin{pmatrix}
\lambda + \varepsilon& &  \\
&   \lambda- \varepsilon & \\
& &  \mu
\end{pmatrix}
$$
and 
$$
A_0 = \begin{pmatrix}
\lambda & 1&  \\
&\lambda  &   \\
& & \mu
\end{pmatrix}.
$$
Denote by $\tau_{A(\varepsilon)}$ and $\tau_{A_0}$ the corresponding conjugating polynomials and by $m_{A(\varepsilon)}$, $m_{A_0}$ the corresponding minimal polynomials.  Then
$$
m_{A(\varepsilon)}(z)= m_{A_0}(z) - \varepsilon^2(z-\mu)
$$
where $m_{A_0}(z)= (z-\lambda)^2(z-\mu)$ while 
\begin{multline}
\tau_{A(\varepsilon)}(z) =  \overline \lambda [1-\frac{(z-\lambda)^2-\varepsilon^2}{(\mu-\lambda)^2-\varepsilon^2}] +   
\frac{\overline {\varepsilon}}{\varepsilon}  (z-\mu) \frac{(z-\lambda)(\lambda-\mu) - \varepsilon^2}{(\mu-\lambda)^2-\varepsilon^2}
\\ +\overline \mu \frac{(z-\lambda)^2-\varepsilon^2}{(\mu-\lambda)^2 - \varepsilon^2}
\end{multline}
and 
\begin{equation}
\tau_{A_0}(z)=\overline \lambda \frac{z-\mu}{\lambda-\mu} [2 - \frac {z-\mu}{\lambda-\mu} ]  + \overline \mu (\frac{z-\lambda}{\mu-\lambda})^2.
\end{equation}
In particular, 
$$
\tau_{A(\varepsilon)}(z)=\tau_{A_0}(z) +   
\frac{\overline {\varepsilon}}{\varepsilon}  \frac{(z-\mu)(z-\lambda)}{\lambda-\mu} + \mathcal O (\varepsilon^2).
$$
Hence,  as $\varepsilon \rightarrow 0$ the polynomial  has discontinuous coefficients, but notice that when evaluated at $A(\varepsilon)$  the discontinuity disappears:
$$
  \tau_{A(\varepsilon)}(A(\varepsilon))=\tau_{A_0}(A(\varepsilon))+ \mathcal O(\varepsilon).
$$
Also notice that $\tau_{A(0)}$ is of first degree:
$
\tau_{A(0)}(z)= \overline \lambda \frac{z-\mu}{\lambda-\mu}+ \overline{\mu}\frac{z-\lambda}{\mu-\lambda}
$.
From these examples we also  see that estimating the size of $\|A^c\|$  using the coefficients $\alpha_j$ is  in general not practical as they can be arbitrarily large  compared with the norm. In fact 
$
\tau_{A_0}(z)= \frac{\overline{\mu - \lambda}}{(\mu-\lambda)^2} z^2 + \cdots
$
and yet at the same time 
$
\|\tau_{A_0}(A_0)\| = \| A_0^c\| = \max\{|\lambda|, | \mu |Ê\}.
$
\end{example}  
 
\section{Approximative computation of the conjugate based on the integral formula }

We consider here approximations of the following  somewhat idealized form:   we assume that the resolvent is given exactly and integrations can be performed exactly except near singularities.  The effect we are interested is  the following.  Assume that one performs a rough search for the $\varepsilon$-pseudospectrum and omits from the area integral  small discs where the resolvent is very large.  The discussion is mainly targeted to analysis, rather than actual numerical computation - for which one should prefer the  modified Schur-Parlett algorithm  before.

Two kinds of errors appear.  Omitting a small disc centered around an eigenvalue, containing no other eigenvalues,  {\it does not} cause an error in the invariant subspace related to the eigenvalue,  but an error appears  in the complementary subspace. And, secondly, in practice, the centers may not be at the eigenvalues and the discs may contain several eigenvalues. Notice that we may shrink the set $\Omega$ into discs containing the spectrum.  The boundary integral and the area integral do depend on $\Omega$ but their sum stays constant.

Suppose eigenvalues $\lambda_1, \cdots, \lambda_s$ are inside   the disc $B=\{\zeta : |\zeta- \lambda| < r\}$ and denote by $P$ the spectral projection such that $AP$ operates in the invariant subspace corresponding to these eigenvalues.  Then we have the following.
\begin{lemma}
\begin{equation}
(AP)^c + \frac {1}{\pi} \int_B (\zeta I-A)^{-1} d\mu(\zeta) = \frac{1}{2 \pi i} \int_{\partial B} \overline \zeta \  (\zeta I- A)^{-1} dÊ\zeta = \overline \lambda P.
\end{equation}
\end{lemma}
\begin{proof}
Integrating over $\partial B$ reduces the operating into the invariant subspace  onto which $P$ projects.  And substituting $\zeta= \lambda + r e^{i\theta}$ and noting that the spectral radius of $(A-\lambda I)P$ is less than $r$ we  obtain
\begin{align}
 &\frac{1}{2 \pi i} \int_{\partial B} \overline \zeta \  (\zeta I- A)^{-1} dÊ\zeta \ P=\\
 & \frac{1}{2 \pi } \int_0^{2 \pi} (\overline \lambda + r e^{-i\theta})\sum_{k=0}^\infty
 \frac{1}{(r e^{i\theta})^{k+1}}((A-\lambda I)P)^k r e^{i\theta} d\theta \ P= \overline \lambda P.
\end{align}
\end{proof}
Suppose now that $P=\sum_{i=1}^s P_i$ where $P_i$ denote  the projections onto the subspaces wrt eigenvalues $\lambda_i$.  Then clearly 
$$
(AP)^c = \sum_{i=1}^s \overline \lambda_i P_i.
$$
Let further  $\sigma(A) \subset \cup_{j=1}^N B(z_j, \varepsilon_j)= \Omega$ where the closures of the discs $B(z_j, \varepsilon_j)$ do not intersect.  Denote 
\begin{equation}
A_\Omega^c =   \frac{1}{2 \pi i} \int_{\partial \Omega} \overline \zeta \  (\zeta I- A)^{-1} dÊ\zeta.
\end{equation}
Then, with $A=TJT^{-1}$ denoting the Jordan form, we obtain the following estimate.
\begin{proposition}
We have in the notation above
\begin{equation}
\| A^c - A_\Omega^c\| \le \varkappa(T)   \ \varepsilon,
\end{equation}
where $\varepsilon=  \max_{1 \le j \le N}  \varepsilon_j.$
\end{proposition}

Thus, if we  take $\Omega$ to be a collection of nonintersecting discs, the boundary integral alone gives an approximation  with error proportional to  the radius of the discs.   Assume now that $\Omega$ is fixed but we  include the area integral as well, however, omitting  small discs around each eigenvalue.  We assume first that these discs are centered at the eigenvalues and each disc contains  only  one  eigenvalue.
Then the error is  proportional to $\varepsilon ^2$. 

 So, let $\varepsilon < \min |\lambda_i - \lambda_j|/2$ so that the discs $B(\lambda_j, \varepsilon _j)$ with $\varepsilon_j <\varepsilon$ do not intersect.
Given  $\Omega$ containing all these discs we denote
$$
\Omega_\varepsilon = \Omega \setminus  \cup_{j=1}^N B(\lambda_j, \varepsilon_j).
$$
 Then
 \begin{equation}
 A^c=  \frac{1}{2 \pi i} \int_{\partial \Omega} \overline \zeta \  (\zeta I- A)^{-1} dÊ\zeta
- \frac{1}{\pi} \int_{\Omega_\varepsilon} (\zeta I - A)^{-1} d\mu(\zeta) 
 -E_\varepsilon
  \end{equation}
  where
  $$
  E_\varepsilon =   \sum_{j=1}^N \frac{1}{\pi} \int_{B(\lambda_j, \varepsilon_j)} (\zeta I - A)^{-1} d\mu(\zeta)
$$
 Here the error term of omitting the discs behaves as $\mathcal O(\varepsilon^2)$ .
    
 \begin{proposition} 
 Assume that in the Jordan canonical form $A=TJT^{-1}$ the Jordan block $J_j= \lambda_j + S_j$ is of  size $1+n_j \times 1+n_j$ with $n_j \ge 0$.  In the notation above  and assuming  $\varepsilon $ is small enough so that the discs do not intersect, we have
 \begin{equation}
 E_\varepsilon =   \sum_{j=1}^N \frac{1}{\pi} \int_{B(\lambda_j, \varepsilon_j)} (\zeta I - A)^{-1} d\mu(\zeta)= T \  [ \ \bigoplus_{j=1}^N \sum_{i=0}^{n_j} \sum_{k\not=j} \frac{\varepsilon_k^2}{(\lambda_k-\lambda_j)^{i+1}}S_j^{i} \ ]\ T^{-1}.
 \end{equation} 
 \end{proposition}
 
 \begin{proof}
It clearly sufficces to consider  the Jordan block $J_j$.   Integral over $B(\lambda_j, \varepsilon_j)$ vanishes, as the  disc is centered at the eigenvalue.  In fact, substituting $\zeta= \lambda_j + t e^{i\theta}$ we have
$$
 \frac{1}{\pi} \int_{B(\lambda_j, \varepsilon_j)} (\zeta I - J_j)^{-1} d\mu(\zeta) = 
 \frac{1}{\pi} \int_0^{\varepsilon_j} \int_0^{2\pi} \sum_{l=0}^{n_j} t^{-l} e^{-i (1+l)\theta }S_j^l  \ d\theta dt=0.
$$
Consider now the integral over $B(\lambda_k,\varepsilon_k)$ with $k\not=j$.
Substituting now $\zeta= \lambda_k + t e^{i\theta}$ we  have  
$$
\frac{1}{\pi} \int_{B(\lambda_k,\varepsilon_k)} \frac{1}{(\zeta-\lambda_j)^{1+i}} \ d\mu(\zeta)
= \frac{\varepsilon_k^2}{(\lambda_k -\lambda_j)^{1+i}}
$$ 
from  which the claim follows by  collecting all terms and summing up. 
 
 \end{proof} 
 
 Our last remark concerns the  error committed by not having the centers of the omitted discs exacty at the eigenvalues.  If the centers  $z_j$ satisfy $|z_j-\lambda_j| \le \delta$  then similar calculations as above show that the extra errors committed are of order $\mathcal O(\delta)$  which allow us to formulate the following.  
\begin{proposition}
Assume there exists $C$ such that  for every $j$ the following holds: $|z_j-\lambda_j|  \le C \varepsilon^2$ ,   $\varepsilon_j <\varepsilon$,  the closures of discs $B(z_j, \varepsilon_j)$  do not intersect each  others while $\lambda_j \in  B(z_j, \varepsilon_j)$.
Then as $\varepsilon \rightarrow 0$, 
\begin{equation}
 E_\varepsilon =   \sum_{j=1}^N \frac{1}{\pi} \int_{B(z_j, \varepsilon_j)} (\zeta I - A)^{-1} d\mu(\zeta)=  \mathcal O(\varepsilon^2).
 \end{equation} 
 \end{proposition}


\section{Absolute value and the polar representation}

 We  define the {\it absolute value } of $A$ as follows. 
\begin{definition}
Given $A\in \mathbb M_n(\mathbb C)$  we denote by   ${\rm abs}(A)$ the matrix satisfying
\begin{equation}
{\rm abs}(A) = (A A^c)^{1/2}
\end{equation}
and call it the absolute value of the matrix $A$.
\end{definition}

To see, that this is well defnied, it clearly sufficces to consider  a  single  $1+m\times 1+m$ Jordan block  $J=\lambda I + S$.   If $\lambda=0$ then $J^c=0$ and so ${\rm abs}(J)=0$.  For $\lambda \not=0$ we have
$$
{\rm abs}(J)^2 = \lambda (I+\frac{1}{\lambda }S) \overline \lambda  I
$$
and thus using binomial expansion and $S^{m+1}=0$ we obtain
$$
{\rm abs}(J)=|\lambda| (I+ \frac{1}{2\lambda} S + \cdots + {{1/2}\choose {m}} \frac{1}{\lambda^m} S^m).
$$
As the mapping $\zeta \mapsto (\zeta \overline \zeta)^{1/2} $ is smooth away from  origin,  we obtain the same result by the calculus of Section 3  for nonsingular matrices.

\begin{remark}
Observe that  if $A$ has a nontrivial Jordan block  for a nonvanishing eigenvalue, then   $\rm{abs}(\rm{abs}(A))\not= \rm{abs}(A)$.  However,  the definition makes sense as it agrees with the Stoke's theorem.  Note also that for complex $\alpha$ we do have $\rm{abs}(\alpha A)= |\alpha| \rm{abs}(A)$ and that for $k\ge1$ we have $\rm{abs}(A^k)= (\rm{abs}(A))^k$.
\end{remark}

\begin{remark}
If $A$ has only real eigenvalues, then it is tempting and possible to consider the function $\xi \mapsto \rm{sign}(\xi) \xi$  which then also leads to a definition of "absolute value" 
of a matrix.   However,  this function is the restriction of 
$$
\zeta= \xi+i\eta  \mapsto \rm{sign}(\xi) \zeta
$$
 to the real line.  But introducing $\rm{sign}(\zeta)  = (\zeta^2)^{-1/2} \zeta$
we notice that   $\rm{sign}(\zeta)\zeta$ is holomorphic away from the imaginary axis (when we cut the plane along the negative real axis) and hence
 the matrix $\rm{sign}(A) A$ is well defined in the holomorphic functional calculus, provided $A$ has no purely imaginary eigenvalues.  The  matrix function $\rm{sign}(A)$ has  received much attention and can be computed by Newton iteration, [16]. 
In particular,  one obtains the commutative {\it matrix sign decomposition}
$$
A= {\rm sign}(A)N
$$
where $N=(A^2)^{1/2}.$

 \end{remark}
Recall next the {\it polar decomposition} of a nonsigular matrix $A$.  
If we put $|A|=(A^*A)^{1/2}$ and then set  $U=|A|^{-1}A$,  the matrix $U$ is unitary and we have
$A=|A| U$.  Recall that  $U$ and $|A|$ commute if and only if $A$ is normal.

Since ${\rm abs}(A)$ is nonsingular if and only if $A$ is, we may in the nonsingular case set
\begin{equation}\label{arg}
V= ({\rm abs}(A))^{-1} A
\end{equation}
and obtain
\begin{equation}
A= {\rm abs}(A)V.
\end{equation}
Using the  functional calculus to the function 
$$
{\rm arg}:  \zeta \mapsto (\zeta \overline \zeta)^{-1/2} \zeta
$$
we clearly obtain $V={\rm arg}(A)$.
In order not to be mixed with the polar decomposition we shall call it the {\it polar representation } of the matrix.
\begin{definition}
Let $A \in \mathbb M_n(\mathbb C)$ be nonsingular.   Then we  say that 
$$
A= {\rm abs}(A) V
$$
is the polar representation of  $A$.  
\end{definition} 
Notice that it is always commutative  and if $A$ is normal we have $|A| = {\rm abs}(A)$ so  it agrees then and only then with the usual polar decomposition.


\begin{example}
Consider the matrix
\begin{equation}
A= \begin{pmatrix}
\alpha&\gamma\\
& \beta
\end{pmatrix}
\end{equation}
Let $\alpha\not=\beta$ be both nonzero. Then the polar representation of $A$ is given by
\begin{equation}
{\rm abs}(A)= \begin{pmatrix}
|\alpha| &\frac{|\alpha|-|\beta|}{\alpha-\beta}\gamma\\
& |\beta|
\end{pmatrix}
\end{equation}
and 
\begin{equation}
V = \begin{pmatrix}
\frac{\alpha}{|\alpha|} &(\frac{\alpha}{|\alpha|}-\frac{\beta}{|\beta|})\frac{\gamma}{\alpha-\beta}\\
& \frac{\beta}{|\beta|}
\end{pmatrix}.
\end{equation}
Notice that if $\gamma\not=0$, the  matrix   ${\rm abs}(A)$ is self-adjoint when $\alpha$ and $\beta$ have the same modulus, while $V$ is unitary when they have the same argument.
When $\alpha=\beta$ we have the following:
\begin{equation}
A= \begin{pmatrix}
\alpha&\gamma\\
& \alpha
\end{pmatrix}
\end{equation}
\begin{equation}
{\rm abs}(A)= \begin{pmatrix}
|\alpha| &\frac{1}{2} \frac{\overline \alpha}{|\alpha|} \gamma\\
& |\alpha|
\end{pmatrix}
\end{equation}
and
\begin{equation}
V= \begin{pmatrix}
\frac{\alpha}{|\alpha|} &\frac{\gamma}{2|\alpha|}\\
& \frac{\alpha}{|\alpha|}
\end{pmatrix}
\end{equation}
We conclude that ${\rm abs}(A)$ is self-adjoint and $V$is unitary  only when $\gamma=0$  and then, of course,  $A$ is normal.
 \end{example}

\section{Consistency  with the multicentric functional calculus}

In [19]  we considered scalar functions $\varphi$ of the form
\begin{equation}\label{multirep}
\varphi(z)= \sum_{j=1} \delta_j(z) f_j(p(z))
\end{equation}
where $f_j$'s are assumed to be continuous  in  a  compact $M\subset \mathbb C$, while $p$ is a polynomial with simple roots $\Lambda= \{\lambda_j\}_{j=1}^d$ and $\delta_j$ denote the Lagrange interpolation polynomials satisfying $\delta_j(\lambda_k)= \delta_{j,k}$.
The  key idea was this:  if $p(A)$ is diagonalizable, then $f_j(p(A))$ is well defined for continuous functions, while $\delta_j(A)$ is defined as   a polynomial.   Thus $\varphi(A)$ can be written as
\begin{equation}
\varphi(A)= \sum_{j=1}^d \delta_j(A) f_j(p(A)).
\end{equation}
This way one can  code the extra smoothness-structure needed  for the functions when applied to matrices with nontrivial Jordan blocks.

Notice that this  functional calculus is continuous at nontrivial Jordan blocks in the following sense. Assume that $A(0)$ has nontrivial Jordan blocks, $A(t)$ is diagonalizable and  $A(t)  \rightarrow A(0)$ as $t\rightarrow 0$.    If  $p$ is such that $p(A(0))$ is diagonalizable, the from the continuity of $f_j$'s we conclude 
$$
\varphi(A(t)) \rightarrow \varphi(A(0)).
$$
It is then of interest to know what is the set of functions $\varphi$  which can be represented in the form (\ref{multirep}).
Without going into details here,  we considered continuous functions  $M \rightarrow \mathbb C^d$ which under a suitable product became a Banach algebra, denoted $C_\Lambda(M)$, and then $\varphi$ could be considered as the Gelfand transform  of $f$, defined on $p^{-1}(M)$.


However,  even so $\tau: z \mapsto \overline z$ is smooth, it  cannot  be represented in the multicentric  form at critical points of the transforming polynomial. Technically, the representations are discontinuous at the critical eigenvalues.

\begin{proposition}
Assume $p$ is a polynomial with distinct zeros and suppose $M \subset \mathbb C$ be compact, containing a critical value of $p$.  
Then  $\tau$ cannot be a Gelfand transform of a function $f\in C_\Lambda(M)$.
\end{proposition}
\begin{proof}
The proof goes by getting a contradiction.  Suppose   there exists   a continuous function $f : M \mapsto \mathbb C^d$ such that for $z\in p^{-1}(M)$
\begin{equation}\label{multitau}
\tau ( z )= \sum_{j=1}^d \delta_j(z) f_j(p(z)).
\end{equation}
Let $\lambda$ be a critical point of $p$ such that $p(\lambda) \in M$ and take 
\begin{equation}
 A(t)= \begin{pmatrix} \lambda+ t e^{i\theta}&1\\
 &\lambda
 \end{pmatrix}.
 \end{equation} 
Thus, $p(A(t)) \rightarrow p(A(0)) = p(\lambda)I$ when $t \rightarrow 0$  and when substituting into (\ref{multitau}) continuity of $f$ would imply $\tau(A(t)) \rightarrow \tau (A(0))$, contradicting Example \ref{eka}.

\end{proof}

Introducing the conjugate in this paper was partly motivated by the fact that even so the multicentric calculus extends the functional calculus considerably it does not reach up to the conjugation.  However, it is important to notice that the calculus presented here and the multicentric calculus are consistent.  To that end recall that by the Stone-Weierstrass theorem polynomials $q(z,\overline z)$ are dense among continuous  functions.  Consider therefore functions $\varphi$ of the form
\begin{equation}\label{perusmulti}
\varphi(z) = \sum_{j=1}^d \delta_j(z)  q_j(p(z), \overline {p(z)})
\end{equation}
in which  we can evaluate  $\varphi(A)$  without the extension of this paper, provided $p(A)$ is diagonalizable.  But if $A=TJT^{-1}$ is a Jordan form with $J=D+N$,  and $p(A)$ is diagonalizable, then we have   $p(A)=T p(D)T^{-1} $ and therefore we would substitute $T\overline {p( D)}T^{-1}$  in place of $\overline {p(z)}$ in (\ref{perusmulti}).  However, this then equals $\overline p (A^c)$. Thus, the consistency follows from the following fact:
$$
\tau(p(A)) = \overline p(\tau (A)).
$$


\bigskip

  {\bf References}
 \bigskip

[1]  Steven R. Bell:  The Cauchy Transform, Potential Theory, and Conformal Mapping, CRC Press, Inc.  1992

[2] A. Buchheim: An extension of a theorem of Professor Sylvester's relating to matrices, Phil.Mag., 22 (135), 173-174, 1886 

[3]  L. Carleson, An interpolation problem for bounded analytic functions ", Amer. J. Math., 80 (1958), 921-930

[4] Narinder S. Claire: Spectral Mapping Theorem for the Davies-Helffer-Sj\"ostrand Functional Calculus, J. Math. Physics, Analysis, Geometry, 2012, vol. 8, No. 3,  221-    239

[5]  Peter Colwell, Blaschke Products:  bounded analytic functions, Univ. of Michigan Press, 1985

[6] H.Garth Dales:  Banach Algebras and Automatic Continuity, Oxford Science Publ. London Math.Soc. Monography, New Series 24, 2000

[7] E.B. Davies, The Functional Calculus. Ñ J. London Math. Soc. 52 (1995), No. 1, 166 - 176

[8] E.B. Davies, Spectral Theory and Differential Operators. Cambridge Studies in Advanced Mathematics, 42, Cambridge Univ. Press, Cambridge, 1995.

[9] E.M. Dynkin: An operator calculus based on the Cauchy-Green formula.
(Russian), Investigations on linear operators and the theory of functions,
III. Zap. Nauv cn. Sem. Leningrad. Otdel. Mat. Inst. Steklov. (LOMI) 30
(1972), 33-39

[10] J. P. Earl,  On the interpolation of bounded sequences by bounded functions, J. London Math. Soc, 2 (1970), 544-548.

[11] J.P. Earl, A note on bounded interpolation in the unit disc, J. London Math. Soc. (2) 13 (1976), 419-423

[12]  G. Frobenius: \"Uber die cogredienten  Transformationen der bilinearen Formen,
Sitzungsberichte der K\"oniglich Preussischen Akademie der Wissenschaften zu Berlin 16: 7-16, 1896.

[13]  F.R. Gantmacher: The Theory of Matrices, Volume one. Chelsea. New York.1959

[14] B. Helffer, J. Sj\"ostrand:  E'quation de Schr\"odinger avec champ magn\'etique et e'quation de Harper,
Lecture Notes in Physics, Vol. 345, Springer, Berlin, 1989, pp. 118 - 197.

[15] Peter Henrici:  Applied and Computational Complex Analysis, Vol 3,  John Wiley \& Sons, 1986

[16]  Nicholas J. Higham:  Functions of Matrices: Theory and Computation.  SIAM
2008 

[17]  Roger A. Horn, Charles R. Johnson:  Topics in Matrix Analysis, Cambridge Univ. Press, 1991

[18] K. Kellay, M. Zarrabi: Normality:  non-quasianalyticity and invariant subspaces, J. Operator Theory 46 (2001), 221-250

[19]  Olavi Nevanlinna:   Polynomial as a new variable - a Banach algebra with functional calculus, arXiv: 1506.00634 [math.FA], submitted on  1 Jun 2015,  to appear   Operators and Matrices. 

[20]   Piotr Niemiec:  Functional calculus for diagonalizable matrices,   Linear and Multilinear Algebra, Vol 62, No 3 (2014), 297-321
 
[21] R. F. Rinehart:  The equivalence of definitions of a matric function, Amer. Math. Monthly 62 (1955), 395 - 414

[22] J. J. Sylvester: On the equation to the secular inequalities in the planetary theory, Philosophical Magazine, 16: 267-269, 1883

\end{document}